\documentclass[10pt, article]{amsart}
\usepackage{tikz}
\usetikzlibrary{calc}
\usepackage{ae} 
\usepackage[T1]{fontenc}
\usepackage[cp1250]{inputenc}
\usepackage{amsmath}
\usepackage{amssymb, amsfonts,amscd,verbatim}

\usepackage[normalem]{ulem}
\usepackage{hyperref}
\usepackage{indentfirst}
\usepackage{latexsym}
\usepackage{mathrsfs} 
\input xy
\xyoption{all}

\usepackage{amsmath}    

\theoremstyle{plain}
\newtheorem{Pocz}{Poczatek}[section]
\newtheorem{Proposition}[Pocz]{Proposition}
\newtheorem{Theorem}[Pocz]{Theorem}
\newtheorem{Corollary}[Pocz]{Corollary}

\newtheorem{Lemma}[Pocz]{Lemma}
\newtheorem{Observation}[Pocz]{Observation}

\newtheorem{Question}[Pocz]{Question}

\newtheorem{Example}[Pocz]{Example}

\theoremstyle{definition}
\newtheorem{Definition}[Pocz]{Definition}

\theoremstyle{remark}
\newtheorem{Remark}[Pocz]{Remark}
\newtheorem{Exercise}[Pocz]{Exercise}

\DeclareMathOperator*{\diam}{diam}

\def\asdim{\mathrm{asdim}}

\def\diam{\mathrm{diam}}
\def\array{\mathrm{array}}

\numberwithin{equation}{section}
\title[Matrix algebra of sets  and variants of decomposition complexity]
{Matrix algebra of sets  and variants of decomposition complexity}

\author{Jerzy Dydak}

\address{University of Tennessee, Knoxville, TN 37996}
\email{jdydak\@@utk.edu}

\date{ \today
}


\begin{document}
\maketitle
\begin{center}
\today
\end{center}

\begin{abstract}
We introduce matrix algebra of subsets in metric spaces and we apply it to improve results of Yamauchi and Davila regarding Asymptotic Property C. Here is a representative result:\\
Suppose $X$ is an $\infty$-pseudo-metric space and $n\ge 0$ is an integer. The \textbf{asymptotic dimension} $\asdim(X)$ of $X$ is at most $n$
if and only if for any real number $r > 0$ and any integer $m\ge 1$ there is
an augmented $m\times (n+1)$-matrix $\mathcal{M}=[\mathcal{B} |\mathcal{A}]$ (that means $\mathcal{B}$ is a column-matrix and $\mathcal{A}$
is an $m\times n$-matrix)  of subspaces of $X$ of scale-$r$-dimension $0$
such that $\mathcal{M}\cdot_\cap \mathcal{M}^T$
is bigger than or equal to the identity matrix and
$B(\mathcal{A},r)\cdot_\cap B(\mathcal{A},r)^T$
is a diagonal matrix.
\end{abstract}

\section{Introduction}

This paper is devoted to \textbf{decomposition complexity} understood as any coarse invariant defined in terms of decomposing spaces into $r$-disjoint subsets. Historically, the first such invariant, namely asymptotic dimension, was introduced by Gromov for the purpose of studying groups using geometric methods (see \cite{Roe}). In Ostrand (\cite{O$_1$} or
\cite{OstrandDimofMetriSpacesHilbert}) formulation (see \cite{BDLM}) it can be defined as follows:

\begin{Definition}
Suppose $X$ is a metric space. The \textbf{asymptotic dimension} of $X$ is at most $n$
if, for every real number $r > 0$, there is a decomposition of $X$ into a union of
its subsets $X_0, \ldots, X_n$ such that each $X_i$ is the union of a uniformly bounded
and $r$-disjoint family $\mathcal{U}_i$. That means there is a real number $S > 0$ with each member
of $\mathcal{U}_i$ being of diameter at most $S$ and the distance between points belonging
to different elements of $\mathcal{U}_i$ is at least $r$.
\end{Definition}

Since then several concepts related to asymptotic dimension were introduced by various authors. One can see them as a spectrum with asymptotic dimension being the strongest concept and weak coarse paracompactness being the weakest (see \cite{CDV3}).
The concept closest to asymptotic dimension was introduced by Dranishnikov \cite{Dra1} under the name of Asymptotic Property C:

\begin{Definition}
Suppose $X$ is a metric space. $X$ has \textbf{Asymptotic Property C}
if, for every sequence of real numbers $r_i > 0$, there is a decomposition of $X$ into a finite union of
its subsets $X_0, \ldots, X_n$ for some natural $n$ such that each $X_i$ is the union of a uniformly bounded
and $r_i$-disjoint family $\mathcal{U}_i$. 
\end{Definition}
A concept weaker than asymptotic dimension appeared in \cite{GTY1} under the name of \textbf{finite
decomposition complexity} (FDC) in order to study questions concerning the topological rigidity of manifolds (see \cite{GTY1} and \cite{GTY2} for details): the Bounded Borel Conjecture (\cite{RY1, RY2}, \cite{GTY1}), the integral Novikov conjecture for the algebraic K-theory of group rings $R[\Gamma]$ (see \cite{RTY}, \cite{BC}, \cite{GYu}, \cite{STG}, \cite{GHW}).

The class of metric spaces with finite decomposition complexity contains all countable linear groups equipped with a proper (left-)invariant metric (see \cite{GTY1}, \cite{GTY2}).

Subsequently, Dranishnikov and Zarichnyi \cite{DZ} introduced a simpler concept, namely \textbf{straight finite decomposition complexity} which is much closer in spirit to the Asymptotic Property C.
That concept was subsequently generalized in \cite{Dyd1} and, independently, in \cite{RR}.
The common feature of all the generalizations is that they imply Property A of G.Yu
(see \cite{NowakYu}, \cite{CDV1}, \cite{CDV2}, and \cite{CDV3}  for various characterizations of it).

This paper introduces new variants of decomposition complexity, namely Asymptotic Property D, which is stronger than Asymptotic Property C and we generalize known results of
T.Yamauchi \cite{Yam} and T.Davila \cite{Davila}. Namely,
T.Yamauchi \cite{Yam} provided an ingenious proof that the infinite direct product of integers has 
Asymptotic Property C. T.Davila \cite{Davila} generalized it to arbitrary reduced product of countable groups of finite asymptotic dimension. Our methods use matrix algebra of arrays of subsets of spaces yielding simpler proofs and stronger results.

See \cite{BG}, \cite{BN}, and \cite{DZ} for other recent results concerning asymptotic property C.

\section{Matrix algebra of sets}
In this section we develop the matrix algebra of set-valued arrays in analogy to the classical matrix algebra of vectors. 

Given a set $X$ we will consider arrays $\mathcal{A}$ consisting of subspaces of $X$ in a similar manner to row-vectors with real-valued coordinates. Another point of view is to consider a subspace-array $\mathcal{A}$ as a function from its index set $S$ to $2^X$, the family of all subsets of $X$.
Each subspace-array $\mathcal{A}$ of $X$ has its \textbf{transpose} $\mathcal{A}^T$.

\begin{Definition}
Given two subspace-arrays $\mathcal{A}$ and $\mathcal{B}$ of $X$ indexed by the same set $S$, the \textbf{$\cap$-dot product} of them is defined as
$$\mathcal{A}\cdot_\cap \mathcal{B}=\bigcup\limits_{s\in S}\mathcal{A}(s)\cap \mathcal{B}(s).$$

The \textbf{union} of $\mathcal{A}$ and $\mathcal{B}$ is defined as the subspace-array sending $s\in S$ to $\mathcal{A}(s)\cup \mathcal{B}(s)$.

The \textbf{set-theoretic norm} of $\mathcal{A}$ is defined as
$$\mathcal{A}\cdot_\cap \mathcal{A}.$$
Thus, $\mathcal{A}$ represents a cover of $X$ if and only if its set-theoretic norm is $X$.
\end{Definition}

Subspace-arrays indexed by the same set $S$ become a \textbf{partially ordered set}:
the inequality $\mathcal{A}\leq \mathcal{B}$ means $\mathcal{A}(s)\subset \mathcal{B}(s)$ for all $s\in S$.

There is a \textbf{function} denoted by $\array(?)$ from $2^X$ to subspace-arrays indexed by $S$. Namely, $\array(Y)$ is the constant array with entries equal to $Y\subset X$. Notice $\array(X)$ is the \textbf{maximal array}.

Subsets $B$ of $X$ can be thought of as scalars. Thus,
$B\cdot \mathcal{A}$ is defined as the subspace-array sending $s\in S$ to $B\cap \mathcal{A}(s)$. Notice the analogs of vector identities hold:
$$C\cdot (\mathcal{A}\cup \mathcal{B})=C\cdot \mathcal{A}\cup C\cdot \mathcal{B},$$
$$(B\cap C)\cdot \mathcal{A}=B\cdot (C\cdot \mathcal{A})$$
and so on.

\begin{Definition}
Given the cartesian product $S\times T$ of two index sets,
a subspace $S\times T$-matrix of $X$ is a subspace array
indexed by $S\times T$. If $\mathcal{A}$ is a subspace $S\times T$-matrix of $X$ and $\mathcal{B}$ is a subspace $T\times R$-matrix of $X$, then the \textbf{matrix $\cap$-product} of $\mathcal{A}$ and $\mathcal{B}$
is the $S\times R$-matrix whose $(s,r)$ coordinate
is the $\cap$-product of the $s$-th row of $\mathcal{A}$
and the $r$-th column of $\mathcal{B}$.
\end{Definition}

Notice $\mathcal{A}\cdot_\cap (\mathcal{B}\cup \mathcal{C})=
(\mathcal{A}\cdot_\cap \mathcal{B})\cup (\mathcal{A}\cdot_\cap \mathcal{C})$
and $(\mathcal{A}\cup \mathcal{B})\cdot_\cap \mathcal{C}=
(\mathcal{A}\cdot_\cap \mathcal{C})\cup ( \mathcal{B}\cdot_\cap \mathcal{C})$ whenever one of the sides is defined.

\begin{Corollary}
The $\cap$-product of matrices is associative:
$$(\mathcal{A}\cdot_\cap \mathcal{B})\cdot_\cap \mathcal{C}=
\mathcal{A}\cdot_\cap (\mathcal{B}\cdot_\cap \mathcal{C})$$
whenever one side is defined.
\end{Corollary}
\begin{proof}
In view of formulae above, it suffices to consider matrices with at most one entry being non-empty as each matrix is a union of such matrices. In that case either both sides are matrices with empty entries or the unique non-empty entry is the intersection of corresponding entries of the three matrices.
\end{proof}

The $\cap$-dot product is mostly useful to study $r$-disjointness, hence
concepts related to asymptotic dimension. There is another analog of dot product, $\times$-dot product that is useful in studying 
$r$-disjointness in cartesian products.

\begin{Definition}
Given two subspace-arrays $\mathcal{A}$ in $X$ and $\mathcal{B}$ in $Y$, both indexed by the same set $S$, the \textbf{$\times$-dot product} of them is defined as
$$\mathcal{A}\cdot_\times \mathcal{B}=\bigcup\limits_{s\in S}\mathcal{A}(s)\times \mathcal{B}(s).$$
\end{Definition}

Similarly to the $\cap$-dot product one can now define the $\times$-dot product of set-valued matrices.

The role of the identity matrix (given an index set $S$) is played by 
the square matrix $\mathcal{I}_X$ whose all diagonal entries are equal to $X$ and all other entries equal the empty set.

\begin{Proposition}
If each column of a matrix $\mathcal{M}$ represents a cover of $X$
and an array $\mathcal{A}$ also represents a cover of $X$,
then $\mathcal{M}\cdot_\cap \mathcal{A}^T$ represents a cover of $X$.
\end{Proposition}
\begin{proof}
Notice that each column of a matrix $\mathcal{M}$ represents a cover of $X$
if and only if $\mathcal{M}^T\cdot_\cap \mathcal{M}\ge \mathcal{I}_X $,
i.e. each diagonal entry of $\mathcal{M}^T \cdot_\cap \mathcal{M}$ equals $X$. Therefore
$$(\mathcal{M}\cdot_\cap \mathcal{A}^T)^T\cdot_\cap (\mathcal{M}\cdot_\cap \mathcal{A}^T) =\mathcal{A}\cdot_\cap (\mathcal{M}^T \cdot_\cap \mathcal{M})\cdot_\cap \mathcal{A}^T\ge$$
$$ \mathcal{A}\cdot_\cap (\mathcal{I}_X)\cdot_\cap \mathcal{A}^T=
\mathcal{A}\cdot_\cap \mathcal{A}^T=X.$$
\end{proof}

A similar result holds also for the $\times$-product but the proof is a bit different.

\begin{Proposition}
If each column of a matrix $\mathcal{M}$ represents a cover of $X$
and an array $\mathcal{A}$ represents a cover of $Y$,
then $\mathcal{M}\cdot_\times \mathcal{A}^T$ represents a cover of $X\times Y$.
\end{Proposition}
\begin{proof}
If $(x,y)\in X\times Y$, then there is an index $s\in S$, $S$ being
the index-set of $\mathcal{A}$, with $y\in \mathcal{A}(s)$.
In turn, there is an index $r\in R$, $R\times S$ being the index set of
$\mathcal{M}$, with $x\in \mathcal{M}(r,s)$.
Now, $(x,y)\in (\mathcal{M}\cdot_\times \mathcal{A}^T)(r)$.
\end{proof}

The following exercises may lead to a more interesting course in set theory.

\begin{Exercise}
Show that an array $\mathcal{A}$ represents a cover of $X$
if and only if the function $f:(2^X)^S\to 2^X$
defined by
$$f(\mathcal{X})=\mathcal{A}\cdot_\cap \mathcal{X}$$
is a surjection.
\end{Exercise}

\begin{Exercise}
Characterize arrays $\mathcal{A}$ representing equivalence relations on $X$ in terms of $\mathcal{A}\cdot_\cap \mathcal{A}^T$
and $\mathcal{A}^T\cdot_\cap \mathcal{A}$.
\end{Exercise}

\begin{Exercise}
Characterize square matrices $\mathcal{M}$ such that
each equation
$$\mathcal{M}\cdot_\cap \mathcal{X}^T=\mathcal{B}^T$$
has a solution and it is unique.
\end{Exercise}

\section{Matrix algebra in metric spaces}

In contrast to most papers on coarse geometry, we do not restrict ourselves to metric spaces only. It is more convenient to consider a wider class of spaces.

\begin{Definition}
An \textbf{$\infty$-pseudo-metric space} $X$ is a set with a distance function $d$ that satisfies weaker axioms than a metric:\\
1. $d(x,x)=0$ and $d(x,y)=d(y,x)\ge 0$ for all $x,y\in X$,\\
2. $d(x,y)$ is allowed to assume the value of $\infty$,\\
3. $d(x,y)\leq d(x,z)+d(z,y)$ if $d(x,z), d(z,y) < \infty$.
\end{Definition}

The reason that $\infty$-pseudo-metric spaces are more useful than metric spaces is that $\infty$-pseudo-metric spaces have better categorical properties. Namely, it is easy to define the disjoint union of $\infty$-pseudo-metric spaces by requiring that distances between points in different summands are equal to $\infty$. Also, one can easily define arbitrary cartesian products of $\infty$-pseudo-metric spaces with either $l_1$-$\infty$-pseudo-metric or the supremum-$\infty$-pseudo-metric.

If $X$ is an $\infty$-pseudo-metric space, then each subspace array of $X$ has its \textbf{metric-norm} defined as supremum of diameters of its coordinates (we assume the diameter of the empty set is $0$).
Thus, $\mathcal{A}$ represents a \textbf{uniformly bounded family} of subsets of $X$ if and only if its metric-norm is finite.

We can apply the usual functions on subsets of $X$ to subspace arrays of $X$. Thus, the \textbf{ball of radius} $0 < r < \infty$, $B(\mathcal{A},r)$, around $\mathcal{A}$ is defined as the array of balls of radius $r$ around each coordinate of $\mathcal{A}$.

\begin{Definition}
Let $r > 0$. Two subsets $C$ and $D$ of an $\infty$-pseudo-metric space $X$ are \textbf{scale-$r$-disjoint} if their $r$-balls are disjoint:
$$B(C,r)\cap B(D,r)=\emptyset.$$
\end{Definition}

\begin{Remark}
Notice the change of definition of $r$-disjointness from standard literature in the above definition. Essentially, it allows better calculations in matrix algebra than the standard one. That's why we emphasize it by the wording of \textbf{scale-$r$-disjoint} instead of \textbf{$r$-disjoint}.
\end{Remark}

\begin{Lemma}\label{rBallsOfDotProductOfMatrices}
Suppose $X$ is an $\infty$-pseudo-metric space and $r > 0$.
If $\mathcal{A}$ is a subspace $S\times T$-matrix of $X$ and $\mathcal{B}$ is a subspace $T\times R$-matrix of $X$, then
$$B(\mathcal{A}\cdot_\cap \mathcal{B},r)\leq B(\mathcal{A},r)\cdot_\cap B(\mathcal{B},r).$$
\end{Lemma}
\begin{proof}
It suffices to consider the case of $\mathcal{A}$ being a row-array
and $\mathcal{B}$ being a column-array.
Notice that $d_X(x,\bigcup\limits_{s\in S}\mathcal{A}(s)\cap \mathcal{B}(s)) < r$ if and only if there is $t\in S$ such that
$d_X(x,\mathcal{A}(t)\cap \mathcal{B}(t)) < r$
which implies $d_X(x,\mathcal{A}(t)) < r$
and $d_X(x,\mathcal{B}(t)) < r$.
\end{proof}

\begin{Lemma}\label{rBallsOfTimesProductOfMatrices}
Suppose $X$ and $Y$ are $\infty$-pseudo-metric spaces and $r > 0$.
If $\mathcal{A}$ is a subspace $S\times T$-matrix of $X$ and $\mathcal{B}$ is a subspace $T\times R$-matrix of $Y$, then
$$B(\mathcal{A}\cdot_\times \mathcal{B},r)\leq B(\mathcal{A},r)\cdot_\times B(\mathcal{B},r)$$
in both the $l_1$-$\infty$-pseudo-metric or the supremum-$\infty$-pseudo-metric.
\end{Lemma}
\begin{proof}
It suffices to consider the case of $\mathcal{A}$ being a row-array
and $\mathcal{B}$ being a column-array.
Notice that $d_{X\times Y}((x,y),\bigcup\limits_{s\in S}\mathcal{A}(s)\times \mathcal{B}(s)) < r$ if and only if there is $t\in S$ such that
$d_{X\times Y}((x,y),\mathcal{A}(t)\times \mathcal{B}(t)) < r$
which implies $d_X(x,\mathcal{A}(t)) < r$
and $d_Y(y,\mathcal{B}(t)) < r$ in both the $l_1$-metric and in the sup-metric on $X\times Y$.
\end{proof}

\begin{Definition}
Suppose $X$ is an $\infty$-pseudo-metric space, $A$ is its subspace, and $r > 0$. The \textbf{array of scale-$r$-components} of $A$ is the array indexed by $X$ whose $x$-th coordinate is the subspace of $X$ consisting of all points $y$ that can be connected to $x$ by a scale-$r$-chain in $A$, i.e. a sequence of points $x_0=y, \ldots, x_n=x$ in $A$ such that $B(x_i,r)\cap B(x_{i+1},r) \ne\emptyset$ for each $0\leq i < n$.

$A$ is of \textbf{scale-$r$-dimension $0$} if its array of scale-$r$-components is of finite metric-norm.

$A$ is of finite \textbf{scale-$r$-dimension} if it can be represented as a finite union of subspaces of scale-$r$-dimension $0$.
\end{Definition}

One can define scale-$r$-disjoint arrays of subspaces as those whose coordinates are scale-$r$-disjoint, i.e. their $r$-balls are disjoint. However, from the point of view of matrix algebra, the following definition makes more sense:

\begin{Definition}
Suppose $X$ is an $\infty$-pseudo-metric space and $r > 0$.
An array $\mathcal{A}$ is \textbf{scale-$r$-disjoint} if
$$B(\mathcal{A},r)^T\cdot_\cap B(\mathcal{A},r)$$
is a diagonal matrix, i.e. its entries off the diagonal are empty.
\end{Definition}

\begin{Definition}
Suppose $X$ is an $\infty$-pseudo-metric space and $r > 0$.
Two arrays $\mathcal{A}$ and $\mathcal{B}$ are \textbf{$r$-orthogonal} if
they are indexed by the same set and the $\cap$-dot product of their $r$-balls is empty.
\end{Definition}

\begin{Lemma}\label{OrthogonalMatrices}
Suppose $X$ is an $\infty$-pseudo-metric space,  $\mathcal{M}$ is a  subspace matrix in $X$, and $r > 0$.
The following conditions are equivalent:\\
1. Each column of $\mathcal{M}$ is scale-$r$-disjoint,\\
2. Rows of $\mathcal{M}$ are mutually $r$-orthogonal,\\
3. $B(\mathcal{M},r)\cdot_\cap B(\mathcal{M}^T,r)$ is a diagonal matrix.
\end{Lemma}
\begin{proof}
1)$\iff$2) as both conditions mean that $r$-balls of different entries in the same column are disjoint. The same is true of 3).

\end{proof}

\begin{Definition}
A matrix satisfying one of the conditions of \ref{OrthogonalMatrices}
is called \textbf{$r$-orthogonal}.
\end{Definition}

\begin{Lemma}
Suppose $X$ is an $\infty$-pseudo-metric space and $r > 0$.
The $\cap$-product of two $r$-orthogonal matrices is $r$-orthogonal.
\end{Lemma}
\begin{proof}
Since $(\mathcal{M}\cdot_\cap \mathcal{A})^T=\mathcal{A}^T\cdot_\cap\mathcal{M}^T,$
$$B(\mathcal{M}\cdot_\cap \mathcal{A},r)\cdot_\cap B((\mathcal{M}\cdot_\cap \mathcal{A})^T,r)=B(\mathcal{M}\cdot_\cap \mathcal{A},r)\cdot_\cap B(\mathcal{A}^T\cdot_\cap\mathcal{M}^T,r)$$
$$\leq B(\mathcal{M},r)\cdot_\cap B(\mathcal{A},r)\cdot_\cap B(\mathcal{A}^T,r)\cdot_\cap B(\mathcal{M}^T,r)$$
$$\leq B(\mathcal{M},r)\cdot_\cap B(\mathcal{M}^T,r)$$
and a matric less than or equal to a diagonal matrix is diagonal.
Notice that we used $\mathcal{D}\cdot_\cap \mathcal{C}\leq \mathcal{C}$ if $\mathcal{D}$ is a diagonal square matrix.
\end{proof}

\begin{Corollary}
Suppose $X$ is an $\infty$-pseudo-metric space and $r > 0$.
If each column of a subspace matrix $\mathcal{M}$ is scale-$r$-disjoint
and
an array $\mathcal{A}$ is scale-$r$-disjoint, then
$\mathcal{M}\cdot_\cap \mathcal{A}^T$ is scale-$r$-disjoint.
\end{Corollary}

\begin{Definition}
Suppose $X$ and $Y$ are sets.
Given
an array $\mathcal{A}$ in $X$ indexed by $S$ and an array $\mathcal{B}$ in $Y$ indexed by $T$, \textbf{the cartesian product}
$\mathcal{A}\times \mathcal{B}$ is the array in $X\times Y$ indexed by $S\times T$ whose value at $(s,t)\in S\times T$
is $\mathcal{A}(s)\times \mathcal{B}(t)$.

\end{Definition}

\begin{Proposition}
Suppose $X$ and $Y$ are $\infty$-pseudo-metric spaces and $r > 0$. If $A\times B\subset X\times Y$, then the scale-$r$-components array of $A\times B$
is less than or equal to the cartesian product of
the scale-$r$-components array of $A$ and the scale-$r$-components array of $B$
in either $l_1$-$\infty$-pseudo-metric or the supremum-$\infty$-pseudo-metric
on $X\times Y$.
\end{Proposition}
\begin{proof}
Notice that any scale-$r$-chain at $(x,y)$ in $A\times B$ projects
to an scale-$r$-chain in $X$ starting from $x$ and to an scale-$r$-chain in $Y$
starting from $y$.
\end{proof}

\begin{Corollary}
Suppose $X$ and $Y$ are $\infty$-pseudo-metric spaces
and $r > 0$. If $X$ is of scale-$r$-dimension $0$
and $Y$ is of scale-$r$-dimension $0$,
then $X\times Y$ is of scale-$r$-dimension $0$
in either $l_1$-$\infty$-pseudo-metric or the supremum-$\infty$-pseudo-metric.
\end{Corollary}
\begin{proof}
Notice the cartesian product of two arrays of finite metric-norm
has finite metric-norm.
\end{proof}

\begin{Lemma}\label{rDim0OfTimesProductOfArrays}
Suppose $X$ and $Y$ are $\infty$-pseudo-metric spaces and $r > 0$.
If
an array $\mathcal{A}$ in $X$ is scale-$r$-disjoint and of scale-$r$-dimension $0$ and an array $\mathcal{B}$ in $Y$ is of scale-$r$-dimension $0$, then
$\mathcal{A}\cdot_\times \mathcal{B}$ is of scale-$r$-dimension $0$.
\end{Lemma}
\begin{proof}
The projections of any scale-$r$-connected subset $C$ of $\mathcal{A}\cdot_\times \mathcal{B}$ are also scale-$r$-connected, so in case of projecting onto $X$, it must be contained in $\mathcal{A}(s)$ for some index $s$ from the index set.
Therefore $C\subset \mathcal{A}(s)\times \mathcal{B}(s)$
and its diameter is bounded by the sum of diameters of
scale-$r$-components arrays of $\mathcal{A}$ and $\mathcal{B}$.
\end{proof}

\begin{Lemma}\label{rDimZeroOfUnionOfSets}
Suppose $X$ is an $\infty$-pseudo-metric space.
If the metric-norm of the scale-$r$-components array
of $A\subset X$ is at most $M$ and the metric-norm of the scale-$(M+2r)$-components array
of $B\subset X$ is at most $s$, then the metric-norm of the scale-$r$-components array
of $A\cup B$ is at most $M+s+2r$.
\end{Lemma}
\begin{proof}
Notice that any scale-$r$-chain contained in either $A$ or $B$ has diameter at most $\max(M,s)$, so of interest are only chains
that meander between $A$ and $B$ and we may restrict ourselves
to scale-$r$-chains starting in $B$. Let $C(b)$ be the scale-$r$-component of $b\in B$. Any scale-$r$-chain starting at $b$ may hop only to
an scale-$r$-component of $A$ that is within $2r$ from $C(t)$.
Therefore the whole scale-$r$-chain may never enter another scale-$r$-component of $B$ and is of diameter at most $M+s+2r$.
\end{proof}

\section{A characterization of asymptotic dimension}

Suppose $(X,d)$ is an $\infty$-pseudo-metric space and $ r > 0$ is a real
number. We can define a new integer-valued pseudo-metric $d_r$ on $X$
by declaring $d_r(x,y)$ to be the length of the shortest scale-$r$-chain joining $x$ and $y$ if $x\ne y$. If such a chain does not exist, we put $d_r(x,y)=\infty$. Notice the identity map from $(X,d_r)$ to $(X,d)$ is bornologous. Also, if $s\ge r > 0$, then the identity map from $(X,d_r)$ to $(X,d_s)$ is $1$-Lipschitz.

The following is an analog of the parallel-perpendicular decomposition in linear algebra:
\begin{Lemma}\label{MainDecompositionLemma}
Suppose $X$ is an $\infty$-pseudo-metric space, $s\ge r > 0$ are real numbers, and $m\ge 1$ is an integer. For every subset $Y$ of $X$ that is of scale-$8(m+1)s$-dimension $0$ there is a subset array
$\mathcal{Y}^\bot=\{Y_i\}_{0\leq i\leq m}$ of $X$ of scale-$s$-dimension $0$ satisfying the following properties:\\
1. $B(\mathcal{Y}^\bot,s)^T\cdot_\cap B(\mathcal{Y}^\bot,s)$ is a diagonal matrix,\\
2. $\mathcal{Y}^\bot$ is $s$-orthogonal to the constant array $\array(Y)$ with entries $Y$,\\
3. For each array $\mathcal{Z}=\{Z_i\}_{0\leq i\leq m}$ of $X$ of scale-$r$-dimension $0$,
the array $\mathcal{Z}\cup \array(Y)$ with entries $ Y\cup Z_i$ can be expressed as the union of
 $\mathcal{Z}\cap\mathcal{Y}^\bot$ and of an array $\mathcal{A}$ of scale-$r$-dimension $0$ that is bigger than or equal to
$\array(Y)$.
\end{Lemma}
\begin{proof}
For each subset $A$ of $X$ and each natural $k$ let the $k$th outer $s$-ring be the set of points that can be reached from $A$ via a scale-$s$-chain of length at most $k+1$ but cannot be reached from $A$ by a scale-$s$-chain of length less than or equal to $k$. In other words, it is the set $B(A,k+2)\setminus B(A,k+1)$ with balls measured in the pseudo-metric $d_s$.

Consider the array $\mathcal{Y}^\bot$ consisting of outer $s$-rings of $Y$ for any $k$ of the form $3\cdot i+3$, $0\leq i\leq m$. The scale-$s$-components of such rings are contained in the corresponding rings of scale-$s$-components of $X_0$, hence 
$\mathcal{Y}^\bot$ is of scale-$s$-dimension $0$.

Look at the scale-$r$-components of $(Y\cup Z_i)\setminus Y_i$ for a fixed $i$. To show they form a uniformly bounded family it suffices to consider only those scale-$r$-components that intersect both $Y$ and $Z_i$.
Also, if a scale-$r$-component $D$ of $(Y\cup Z_i)\setminus Y_i$
intersects only one scale-$s$-component $E$ of $Y$,
then $D$ has its diameter bounded by the sum of $r$, $\diam(E)$, and the metric-norm of scale-$r$-components array of $Z_i$.
It remains to show that there is no scale-$r$-chain
in $(Y\cup Z_i)\setminus Y_i$ joining points belonging to
different scale-$s$-components of $Y$. Yes, it is so as any such chain
would miss $Y_i$, a contradiction.
\end{proof}

\begin{Theorem}\label{MainCharOfAsdim}
Suppose $X$ is an $\infty$-pseudo-metric space and $n\ge 0$ is an integer. The \textbf{asymptotic dimension} $\asdim(X)$ of $X$ is at most $n$
if and only if for any real number $r > 0$ and any integer $m\ge 1$ there is
an augmented $m\times (n+1)$-matrix $\mathcal{M}=[\mathcal{B} |\mathcal{A}]$ (that means $\mathcal{B}$ is a column-matrix and $\mathcal{A}$
is an $m\times n$-matrix)  of subspaces of $X$ of scale-$r$-dimension $0$
such that $\mathcal{M}\cdot_\cap \mathcal{M}^T$
is bigger than or equal to the identity matrix and
$B(\mathcal{A},r)\cdot_\cap B(\mathcal{A},r)^T$
is a diagonal matrix.
\end{Theorem}

\begin{Remark}
Typically, in matrix algebra, augmented matrices have the last column as the distinguished one. In our case, we find it more convenient to distinguish the first column in augmented matrices.
\end{Remark}

\begin{Observation}
The meaning of \ref{MainCharOfAsdim}
is as follows: each row of matrix $\mathcal{M}$ represents
a covering of $X$, each column, with the exception of the first one, has entries that are mutually $r$-disjoint, and all entries of $\mathcal{M}$ are of scale-$r$-dimension $0$.
\end{Observation}

\begin{proof} (of Theorem \ref{MainCharOfAsdim})
Applying $m=1$ one gets one direction of \ref{MainCharOfAsdim}.

To show the other direction, choose a disjoint array $\{X_i\}_{i=0}^n$
of scale-$8(m+1)r$-dimension $0$ covering $X$ and proceed by induction using \ref{MainDecompositionLemma} for $s=r$.
In the first step we construct the first column of $\mathcal{A}$
and a temporary column $\mathcal{B}$
using $X_1$ and $X_0$ as in \ref{MainDecompositionLemma}.
Thus the first column of $\mathcal{A}$ is $r$-orthogonal to
$\array(X_1)$ and $\mathcal{B}$ is its complement
in $\array(X_0\cup X_1)$.
At each step $i$, $n > i \ge 2$, we choose $\mathcal{X}_{i+1}^\bot$
in $\bigcup\limits_{j=0}^{i+1} X_j$ as in \ref{MainDecompositionLemma}
and set the $i$-th column of $\mathcal{A}$ to be
the array whose $j$-th coordinate is the intersection
of $ \mathcal{X}_{i+1}^\bot(j)$ and $\mathcal{B}(j)$.
Then we adjust $\mathcal{B}$ so that the new $j$-th coordinate
is $(X_{i+1}\cup \mathcal{B}(j))\setminus \mathcal{A}(j,i)$. That way $\mathcal{B}$ always remains to be scale-$r$-dimension $0$.
In step $(n-1)$ we get the desired matrices.
\end{proof}

\section{APD profiles}

\begin{Definition}
Suppose $X$ is an $\infty$-pseudo-metric space.
A finite array of functions $(\alpha_0,\ldots,\alpha_k)$ from
$[0,\infty)$ to $[0,\infty)$ is an \textbf{APD profile} of $X$ if
the following conditions are satisfied:\\
1.
$\alpha_0$ is constant,\\
2. each function $ \alpha_i$, $0\leq i\leq k$, is non-decreasing,\\
3. for any non-decreasing array $(r_0,\ldots,r_{k})$ of positive real numbers
there is a decomposition of $X$ as the union of
its subsets $X_0,\ldots, X_k$ such that each $X_i$, $0\leq i\leq k$,
has scale-$r_{i}$-dimension at most $\alpha_i(r_{i-1})-1$.
\end{Definition}

\begin{Remark}
Notice that $r_{-1}$ in the above definition is undefined but it does not matter as $\alpha_0$ is a constant function. Also,
saying that a space $X$ is of dimension at most $d$, in case $d$ is not a natural number, means that the dimension of $X$ is at most the largest non-negative integer $n$ so that $n < d$.
\end{Remark}

\begin{Example}
Suppose $X$ is an $\infty$-pseudo-metric space.
Its asymptotic dimension is at most $n$ if and only if
 $(1,n)$ is an APD profile of $X$.
\end{Example}

\begin{Example}
Suppose $X$ is an $\infty$-pseudo-metric space.
Its asymptotic dimension is finite if and only if
it has an APD profile consisting of constant functions.
\end{Example}

\begin{Observation}
The most important APD profiles are of the form $(\alpha_0,\ldots,\alpha_k)$,
where $\alpha_0\equiv 1$. Indeed, if $(\alpha_0,\ldots,\alpha_k)$
is an APD profile of $X$, then so is 
$(1,\alpha_0-1,\ldots,\alpha_k)$.
\end{Observation}
\begin{proof}
Indeed, given a non-decreasing array $(r_0,\ldots,r_{k+1})$,
one picks a decomposition of $X$ for the array $(r_1,\ldots,r_{k+1})$
and functions $(\alpha_0,\ldots,\alpha_k)$. Notice that decomposition
works for the array $(r_0,\ldots,r_{k+1})$ and functions $(1,\alpha_0-1,\ldots,\alpha_k)$ as well.
\end{proof}

Quite often it is convenient to deal only with natural numbers:
\begin{Definition}
Suppose $X$ is an $\infty$-pseudo-metric space.
A finite array of functions $(\alpha_0,\ldots,\alpha_k)$ from
positive natural numbers to positive natural numbers is an \textbf{integral APD profile} of $X$ if
the following conditions are satisfied:\\
1.
$\alpha_0$ is constant,\\
2. each function $ \alpha_i$, $0\leq i\leq k$, is non-decreasing,\\
3. for any non-decreasing array $(r_0,\ldots,r_{k})$ of positive natural numbers
there is a decomposition of $X$ as the union of
its subsets $X_0,\ldots, X_k$ such that each $X_i$, $0\leq i\leq k$,
has scale-$r_{i}$-dimension at most $\alpha_i(r_{i-1})-1$.
\end{Definition}

\begin{Proposition}
Suppose $X$ is an $\infty$-pseudo-metric space.
$X$ has an integral APD profile if and only if it has an APD profile.
\end{Proposition}
\begin{proof}
Given any APD profile of $X$ one can easily create an integral APD profile of $X$ by taking integer parts of the functions.

Conversely, given an integral APD profile of $X$, $(\alpha_0,\ldots,\alpha_k)$, one constructs an APD profile of $X$
by applying functions in $(\alpha_0,\ldots,\alpha_k)$ to
integer parts of $(r_0,\ldots,r_{k})$ plus $1$.
\end{proof}

\begin{Proposition}
Suppose $f:X\to Y$ is a coarse embedding of $\infty$-pseudo-metric spaces
and $\beta:[0,\infty)\to [0,\infty)$ is a non-decreasing function such that
$d_X(x,y)\leq r < \infty$ implies $d_Y(f(x),(y))\leq \beta(r)$ for all $x,y\in X$. If $(\alpha_0,\ldots,\alpha_k)$ is an APD profile of $Y$,
then $(\alpha_0\circ\beta,\ldots,\alpha_k \circ\beta)$ is an APD profile of $X$.
\end{Proposition}
\begin{proof}
Suppose $(r_0,\ldots,r_{k})$ is a non-decreasing array of positive real numbers.
Define a new array $(s_0,\ldots,s_{k})$ as the image of $(r_0,\ldots,r_{k})$
under $\beta$.
Choose a decomposition of $Y$ as the union of
its subsets $Y_0,\ldots, Y_k$ such that each $Y_i$, $0\leq i\leq k$,
has scale-$s_{i}$-dimension at most $\alpha_i(s_{i-1})-1$.
Notice the point-inverse under $f$ of a set
of scale-$s_{i}$-dimension $0$ is of scale-$r_{i}$-dimension $0$.
Therefore we get a decomposition of $X$ as the union of
its subsets $X_0=f^{-1}(Y_0),\ldots, X_k=f^{-1}(Y_k)$ such that each $X_i$, $0\leq i\leq k$,
has scale-$r_{i}$-dimension at most $\alpha_i(\beta(r_{i-1}))-1$.
\end{proof}

\begin{Corollary}
Having an APD profile is a hereditary coarse invariant.
\end{Corollary}

\begin{Remark}
Notice that the minimal length of APD profiles is also a hereditary coarse invariant which may be considered as the asymptotic dimension of higher order.
\end{Remark}

\begin{Corollary}
Suppose $(\alpha_0,\ldots,\alpha_k)$ is an APD profile of $X$.
If $\beta:[0,\infty)\to [0,\infty)$ is a non-decreasing function such that
$\beta(a+b)\leq \beta(a)+\beta(b)$ for all $a,b\in [0,\infty)$
and $\lim\limits_{t\to\infty}\beta(t)=\infty$,
then $X$ is coarsely equivalent to $X$ with a new $\infty$-pseudo-metric with APD profile $(\alpha_0\circ\beta,\ldots,\alpha_k \circ\beta)$.
\end{Corollary}
\begin{proof}
The new $\infty$-pseudo-metric on $X$ is composition of the original $\infty$-pseudo-metric with $\beta$ and the identity map between them is a coarse equivalence.
\end{proof}

\begin{Definition}\label{AsymptoticPropertyCDef}
Suppose $X$ is an $\infty$-pseudo-metric space.
$X$ has \textbf{Asymptotic Property C} if for each 
sequence $\{r_i\}_{i\ge 1}$ there is a natural number $n$ and a decomposition
$X=\bigcup\limits_{i=1}^m X_i$ such that each $X_i$ is of scale-$r_i$-dimension $0$.
\end{Definition}

\begin{Definition}
Suppose $X$ is an $\infty$-pseudo-metric space.
$X$ has \textbf{Asymptotic Property D} if $X$ has an APD profile.
\end{Definition}

\begin{Proposition}
Suppose $X$ is an $\infty$-pseudo-metric space.
If $X$ has Asymptotic Property D, then $X$ has Asymptotic Property C.
\end{Proposition}
\begin{proof}
Pick an APD profile $(\alpha_0,\ldots,\alpha_k)$ for $X$.
Given a sequence $\{r_i\}_{i\ge 1}$ of real numbers it suffices to consider the case of it consisting of natural numbers and being increasing. Define a new array $\{s_i\}_{i\ge 1}^k$ as follows: $s_1=r_{\alpha_0(1)}$, $s_{i+1}=r_{p(i)}$, where $p(i)=\sum\limits_{j=1}^i\alpha_j(s_j)$, for $i < k$. Choose a decomposition of $X$ as the union of
its subsets $X_0,\ldots, X_k$ such that each $X_i$, $0\leq i\leq k$,
has scale-$s_{i}$-dimension at most $\alpha_i(s_{i-1})-1$.
In turn, each $X_i$ decomposes into $\alpha_i(s_{i-1})$-sets of scale
scale-$s_{i}$-dimension $0$. All those sets can be enumerated
in such a way that the $i$-th set is of scale-$r_{i}$-dimension $0$
for $i\leq p(k)$.
\end{proof}

\begin{Question}
Is there a metric space having Asymptotic Property C but not having Asymptotic Property D?
\end{Question}

\begin{Proposition}
Suppose $X$ is an $\infty$-pseudo-metric space. The following conditions are equivalent:\\
1. $X$ has Asymptotic Property D,\\
2. There is $k\ge 0$ and an array of integer-valued functions $(\alpha_0,\ldots,\alpha_k)$ such that $\alpha_i$ is defined on non-decreasing
sequences of non-negative integers $r_0,\ldots,r_{i-1}$,
$\alpha_0$ is constant and
for any non-decreasing sequence of natural numbers $r_0,\ldots,r_{k}$
there is a decomposition of $X$ as the union of
its subsets $X_0,\ldots, X_k$ such that each $X_i$, $0\leq i\leq k$,
is of scale-$r_{i}$-dimension at most $\alpha_i(r_0,\ldots,r_{i-1})-1$.
\end{Proposition}
\begin{proof} 1)$\implies$2). Use an integral APD profile of $X$.\\
2)$\implies$1). Define $\beta_0=\alpha_0$.
Given functions $\beta_i$, $i < k$, from natural numbers to natural numbers, define $\beta_{i+1}(r)$
as the maximum of all numbers
$\alpha_i(r_0,\ldots,r_{i-1})$, where $r_j\leq \beta_j(r)$ for $j\leq i-1$.
Notice $(\beta_0,\ldots,\beta_k)$ is an integral APD profile of $X$.
\end{proof}

\begin{Theorem}
Suppose $X$ and $Y$ are subspaces of an $\infty$-pseudo-metric space $Z$. If $(1,\alpha)$ is an APD profile of $X$ and $(1,\beta)$ is an APD profile of of $Y$,
then $(2,\max(\alpha,\beta))$ is an APD profile of of $X\cup Y$.
\end{Theorem}
\begin{proof}
Given $r\leq s$ we may assume $\alpha(r)\leq \beta(r)$
and pick a decomposition $X=X_0\cup\ldots \cup X_p$,
$p=\lfloor\alpha(r)\rfloor$, where $X_0$ is of scale-$r$-dimension $0$
and each $X_i$, $i > 0$, is of scale-$s$-dimension $0$.
Now, pick a decomposition $Y=Y_0\cup\ldots \cup Y_q$,
$q=\lfloor\beta(r)\rfloor$, where $Y_0$ is of scale-$r$-dimension $0$
and each $Y_i$, $i > 0$, is of scale-$(M+2s)$-dimension $0$,
where the metric-norm of the $s$-components array
of each $X_i$, $i> 0$, is at most $M$.
Applying \ref{rDimZeroOfUnionOfSets} we get that
each $X_i\cup Y_i$, $1\leq i\leq p$, is of scale-$s$-dimension $0$.
Those sets replace old $X_i$ and we get a desired decomposition
of $X\cup Y$ into $2+q$ sets.
\end{proof}

An analogous proof to that of Theorem \ref{MainCharOfAsdim} gives the following:

\begin{Theorem}\label{MainCharOfBasicAPDProfile}
Suppose $X$ is an $\infty$-pseudo-metric space and $\alpha$ is a non-decreasing, non-negative-integer-valued function on integers. 
$(1,\alpha)$ is an integral APD profile of $X$
if and only if for any real number $r > 0$ and any integer $m\ge 1$ there is
an augmented $m\times (\alpha(r)+1)$-matrix $\mathcal{M}=[\mathcal{B} |\mathcal{A}]$ (that means $\mathcal{B}$ is a column-matrix and $\mathcal{A}$
is an $m\times \alpha(r)$-matrix)  of subspaces of $X$ of scale-$r$-dimension $0$
such that $\mathcal{M}\cdot_\cap \mathcal{M}^T$
is bigger than or equal to the identity matrix and
$B(\mathcal{A},r)\cdot_\cap B(\mathcal{A},r)^T$
is a diagonal matrix.
\end{Theorem}

\section{APD profiles of products}

\begin{Lemma}\label{TwoDimensionLemma}
Suppose $X$ is an $\infty$-pseudo-metric space of finite scale-$s$-dimension for each $s > 0$.
If $X$ is of scale-$r$-dimension $0$ for some $r > 0$, then
for each $s > 0$ there is a scale-$r$-disjoint array
 $X_0,\ldots, X_p$ of scale-$s$-dimension $0$
representing a cover of $X$.
\end{Lemma}
\begin{proof}
Let $M$ be the metric-norm of the $r$-components array
of $X$. Choose an array $Y_0,\ldots, Y_p$ covering $X$ of scale-$(M+2s+2r)$-dimension $0$.
Using \ref{rDimZeroOfUnionOfSets} we can see that the union $Z_i$ of 
all scale-$r$-components of $X$ intersecting $Y_i$
is of scale-$s$-dimension $0$.
Now, defining $X_0$ to be $Z_0$ and $X_i=Z_i\setminus \bigcup\limits_{j=0}^{i-1}Z_j$ for $i > 0$ leads to a desired array.
\end{proof}

\begin{Proposition}\label{MainAPDProposition}
Suppose $X$ is an $\infty$-pseudo-metric space of 
integral APD profile $(1,\alpha)$ and $Y$ is an $\infty$-pseudo-metric space of finite scale-$s$-dimension for each $s > 0$.
If $Y$ is of scale-$r$-dimension $0$ for some $r > 0$, then
 $X\times Y$ decomposes (in either $l_1$-$\infty$-pseudo-metric or the supremum-$\infty$-pseudo-metric)
into subsets $Z_0, Z_{1}$ such that $Z_0$ is of scale-$r$-dimension $0$
and $Z_1$ is of scale-$s$-dimension at most $\alpha(r)-1$.
\end{Proposition}
\begin{proof}
Put $n=\alpha(r)$.
Decompose $Y$ into subsets $Y_1,\ldots, Y_p$, each of scale-$s$-dimension $0$
such that they are mutually scale-$r$-disjoint using \ref{TwoDimensionLemma}. Pick
an augmented $p\times (n+1)$-matrix $\mathcal{M}=[\mathcal{B} |\mathcal{A}]$ of subspaces of $X$ of scale-$r$-dimension $0$
such that $\mathcal{M}\cdot_\cap \mathcal{M}^T$
has the diagonal consisting of $X$'s and
$B(\mathcal{A},r)\cdot_\cap B(\mathcal{A},r)^T$ is diagonal (see
\ref{MainCharOfBasicAPDProfile}).
Look at $\mathcal{Z}:=(\mathcal{M}^T\cdot_\times \mathcal{Y}^T)^T$.
It is an array that has $n+1$ coordinates, the first one is of
scale-$r$-dimension $0$ and the last $n$ of them are 
of scale-$s$-dimension $0$ by applying \ref{rDim0OfTimesProductOfArrays}.
\end{proof}

\begin{Theorem}
Suppose $X_1,\ldots, X_m,\ldots$ is a sequence of $\infty$-pseudo-metric spaces of finite asymptotic dimension.
If $X_i$ is $2i$-discrete for each $i\ge 1$, then 
$(1,\alpha(k))$ is an integral APD profile of $\prod\limits_{i=1}^\infty X_i$, where $\alpha(k)$ is the asymptotic dimension of
$\prod\limits_{i=1}^k X_i$, in either $l_1$-$\infty$-pseudo-metric or the supremum-$\infty$-pseudo-metric.
\end{Theorem}
\begin{proof}
Suppose $k\ge 1$ and $s > 0$ be an integer. Let $Y:= \prod\limits_{i=k+1}^s X_i$. By
\ref{MainAPDProposition}, $\prod\limits_{i=1}^s X_i$ has a decomposition
into two subsets: $Z_0$ of scale-$k$-dimension $0$,
and $Z_1$ of scale-$s$-dimension at most $\alpha(k)-1$.
Taking cartesian products of these sets with one-point sets
in $\prod\limits_{i=s+1}^\infty X_i$ gives a desired decomposition of
$\prod\limits_{i=1}^\infty X_i$.
\end{proof}

\begin{Remark}
The above theorem not only generalizes Yamauchi's result \cite{Yam} that the infinite direct sum of integers has Asymptotic Property C, it actually says that it has a linear APD profile.
\end{Remark}

\begin{Theorem}\label{ProductOfSpacesOfBasicAPD}
Suppose $X$ and $Y$ are $\infty$-pseudo-metric spaces. If $(1,\alpha)$ is an APD profile of $X$ and $(1,\beta)$ is an APD profile of of $Y$,
then $(2,\alpha\cdot\beta+\alpha+\beta)$ is an APD profile of of $X\times Y$.
\end{Theorem}
\begin{proof}
Suppose $s\ge r > 0$.
Express $X$ as $X_0\cup X_1$, where $X_0$ is of scale-$r$-dimension $0$ and $X_1$ is of scale-$s$-dimension at most
$\alpha(r)-1$.
Express $Y$ as $Y_0\cup Y_1$, where $Y_0$ is of scale-$r$-dimension $0$ and $Y_1$ is of scale-$s$-dimension at most
$\beta(r)-1$. Notice $X_1\times Y_1$ is of scale-$s$-dimension
at most $\alpha(r)\cdot \beta(r)-1$.
Using \ref{MainAPDProposition} we can express
$X\times Y_0$ as a union of one set that is scale-$r$-dimension
$0$ and $\alpha(r)$ sets of scale-$s$-dimension $0$.
A similar observation applies to $X_0\times Y$,
so altogheter we can express $X\times Y$ as a union of $2$ sets
of scale-$r$-dimension $0$
and $\alpha(r)\cdot\beta(r)+\alpha(r)+\beta(r)$ sets
of scale-$s$-dimension $0$.
\end{proof}

A similar proof to that of \ref{ProductOfSpacesOfBasicAPD} gives the following:
\begin{Theorem}
Suppose $X$ and $Y$ are $\infty$-pseudo-metric space.
If both $X$ and $Y$ have Asymptotic Property D, then so does
$X\times Y$ in either $l_1$-$\infty$-pseudo-metric or the supremum-$\infty$-pseudo-metric.
\end{Theorem}

\section{Applications}
In this section we apply our results to generalize previously known theorems for asymptotic dimension or Asymptotic Property C.

\begin{Theorem}
Suppose $G$ is a countable group equipped with a proper left-invariant metric.
If there is an APD profile $(\alpha_0,\ldots,\alpha_k)$ common to all finitely generated subgroups of $G$, then $(\alpha_0,\ldots,\alpha_k)$ is an APD profile of $G$.
\end{Theorem}
\begin{proof}
Given a non-decreasing array $(r_0,\ldots,r_{k})$ of positive real numbers
choose a natural number $M\ge r_{k}$ and let
$H$ be the subgroup of $G$ generated by $B(1_G,2M+2)$.
Choose a decomposition of $H$ as the union of
its subsets $C_0,\ldots, C_k$ such that each $C_i$, $0\leq i\leq k$,
has scale-$r_{i}$-dimension at most $\alpha_i(r_{i-1})-1$.
Now, choose representatives $\{g_s\}_{s\in S}$ of cosets $g\cdot H$. That way each $g\in G$ has a unique decomposition as
$g=g_s\cdot h$, $h\in H$. Also, if $s\ne t$,
then sets $g_t\cdot C_i$ and $g_s\cdot C_i$ are scale-$M$-disjoint.
Indeed, $d(g_th_1,g_sh_2) < 2M+2$ implies
$d(h_2^{-1}g_s^{-1}g_th_1) < 2M+2$ and 
$ h_2^{-1}g_s^{-1}g_th_1\in H$ resulting
in $g_s^{-1}g_t\in H$, a contradiction.

Finally, that means sets $D_i:=\bigcup\limits_{s\in S}g_s\cdot C_i$, $0\leq i\leq k$, form a decomposition of $G$
and are of scale-$r_{i}$-dimension at most $\alpha_i(r_{i-1})-1$.
\end{proof}

\begin{Corollary}[J.Smith \cite{Smith}]
Suppose $G$ is a countable group equipped with a proper left-invariant metric and $n\ge 0$. If all finitely generated subgroups of $G$ are of asymptotic dimension at most $n$, then $\asdim(G)\leq n$.
\end{Corollary}

\begin{Question}
Suppose $G$ is a countable group equipped with a proper left-invariant metric. If all finitely generated subgroups of $G$ are of finite asymptotic dimension, then does $G$ have Asymptotic Property D or Asymptotic Property C?
\end{Question}

\subsection{Asymptotic products}
In \cite{DydLogan} the concept of asymptotic products was introduced that seems to be dual to asymptotic cones of M.Gromov
(see \cite{GroNilp}, \cite{Dries}, and \cite{Kap}).

\begin{Definition}
Suppose $D$ is an infinite countable set. A function
$\alpha:D\to (0,\infty)$ \textbf{ has limit infinity at infinity} if
for each $M > 0$ there is a finite subset $E$ of $D$ such that
$\alpha(d) > M$ for all $d\in D\setminus E$.
\end{Definition}

\begin{Definition}
Suppose $D$ is an infinite countable set and a function
$\alpha:D\to (0,\infty)$ has limit infinity at infinity. Given $\infty$-pseudo-metric spaces $(X_d,\rho_d)$, $d\in D$, the \textbf{asymptotic product}
$(\prod\limits_{d\in D} X_d,\rho_\alpha)$ is
 the cartesian product $\prod\limits_{d\in D} X_d$
equipped with the $\infty$-pseudo-metric $\rho_\alpha$ defined as follows:\\
Given $u,v\in \prod\limits_{d\in D} X_d$, $\rho_\alpha(u,v)$
is the sum $\sum\limits_{d\in D} r_d$, where $r_d$
is equal to $\alpha(d)$ if $0 \leq \rho_d(u(d),v(d))\leq \alpha(d)$
and $u(d)\ne v(d)$,
$r_d=\rho_d(u(d),v(d))$ if $\rho_d(u(d),v(d)) > \alpha(d)$,
and $r_d=0$ if $u(d)=v(d))$.
\end{Definition}

\begin{Corollary}
Given two asymptotic products 
$(\prod\limits_{d\in D} X_d,\rho_\alpha)$ and $(\prod\limits_{d\in D} X_d,\rho_\beta)$
of $\infty$-pseudo-metric spaces $(X_d,\rho_d)$, $d\in D$, the identity function between them is a coarse equivalence.
\end{Corollary}

\begin{Corollary}\label{AsymptoticProductOfAsdimSpaces}
Asymptotic products 
$(\prod\limits_{d\in D} X_d,\rho_\alpha)$ of $\infty$-pseudo-metric spaces
of finite asymptotic dimension have Asymptotic Property C.
\end{Corollary}

\subsection{Reduced products}

\begin{Definition}
Suppose $D$ is an infinite countable set and a function
$\alpha:D\to [0,\infty)$ has limit infinity at infinity. Given $\infty$-pseudo-metric spaces $(X_d,\rho_d)$, $d\in D$, the \textbf{reduced product}
$(\times_{d\in D} X_d,\rho_\alpha)$ is
 the cartesian product $\prod\limits_{d\in D} X_d$
equipped with the $\infty$-pseudo-metric $\rho_\alpha$ defined as the sum 
$$\sum\limits_{d\in D} \alpha(d)\cdot \rho_d(u(d),v(d)).$$
\end{Definition}

\begin{Remark}
T.Davila \cite{Davila} introduced reduced products in the special case
of $D$ being the set of natural numbers and $\alpha(d)=2^d$.
Also, he operates in the metric spaces, so his reduced products
are really \textbf{based reduced products} consisting of points
in our reduced products whose distances to a fixed base point are finite.
\end{Remark}

\begin{Proposition}
Suppose 
$(\prod\limits_{d\in D} X_d,\rho_\alpha)$ is an asymptotic product
of $\infty$-pseudo-metric spaces $(X_d,\rho_d)$, $d\in D$.
If there is $c > 0$ such that each $X_d$ is $c$-discrete (i.e. $\rho_d(x,y)\ge c$ if $x\ne y\in X_d$), then the identity function from the asymptotic product to the reduced product
$(\times_{d\in D} X_d,\rho_\alpha)$ is a coarse equivalence.
\end{Proposition}

\begin{Corollary}
Given two reduced products 
$(\times_{d\in D} X_d,\rho_\alpha)$ and $(\times_{d\in D} X_d,\rho_\beta)$
of $\infty$-pseudo-metric spaces $(X_d,\rho_d)$, $d\in D$, the identity function between them is a coarse equivalence provided there is $c > 0$ such that each $X_d$ is $c$-discrete.
\end{Corollary}

The following result was proved by T.Davila \cite{Davila} in case of based reduced products of either groups or simplicial trees.
\begin{Corollary}\label{ReducedProductOfAsdimSpaces}
Reduced products 
$(\times_{d\in D} X_d,\rho_\alpha)$ of $\infty$-pseudo-metric spaces
of finite asymptotic dimension have Asymptotic Property C provided there is $c > 0$ such that each $X_d$ is $c$-discrete.
\end{Corollary}

\end{document}